\documentclass[12pt]{amsart}
\usepackage{amsfonts}
\usepackage{amsthm}
\usepackage{amsmath}
\usepackage{amssymb}
\usepackage{hyperref}
\usepackage{graphicx}
\usepackage{comment}
\usepackage[capitalise]{cleveref}

\providecommand{\U}[1]{\protect\rule{.1in}{.1in}}
\newtheorem{set}{set}[section]
\newtheorem{theorem}[set]{Theorem}
\theoremstyle{plain}

\newtheorem{corollary}[set]{Corollary}

\newtheorem{lemma}[set]{Lemma}
\newtheorem*{lemma*}{Lemma}

\newtheorem*{problem*}{Problem}
\newtheorem{proposition}[set]{Proposition}

\newtheorem*{observation*}{Observation}
\theoremstyle{definition}
\newtheorem{definition}[set]{Definition}
\newtheorem{example}[set]{Example}

\newcommand{\integers}{\mathbb{Z}}
\allowdisplaybreaks[4]

\newcommand{\lamptwo}{\mathcal{L}_2}
\newcommand{\Z}{\mathbb{Z}}
\newcommand{\N}{\mathbb{N}}
\DeclareMathOperator{\GenCon}{GenCon}
\DeclareMathOperator{\sym}{Sym}

\newtheorem{innercustomthm}{}
\newenvironment{customthm}[1]
{\renewcommand\theinnercustomthm{#1}\innercustomthm}
{\endinnercustomthm}

\title{Medium-scale curvature at larger radii in finitely generated groups}
\author{Robert Kropholler and Brendan Mallery}

\begin{document}

\maketitle
\begin{abstract}
	In this paper we study medium scale curvature, a notion of Ricci curvature for groups. 
	We show that dead-end elements yield non-negative curvature. 
	We make use of related elements to find large classes of positive curvature in the lamplighter and Houghton's group. 
	We also study the effect of increasing the radius of curvature and give examples of positive curvature for arbitrary radius.
	Finally we make some comparisons between this notion of curvature and the Ricci curvature introduced by Ollivier.
\end{abstract}
\section{Introduction}
Various notions of curvature have played a prominent role in geometric group theory. These notions ranges from large scale notions of hyperbolicity \cite{gromov, ghys} to small scale notions such as non-positive curvature \cite{bridson}. 
In \cite{medscale}, a new notion of curvature for a group was defined.
This notion was inspired Ollivier's definition of discrete Ricci curvature \cite{ollivier}. 
In \cite{ollivier}, the classical Ricci curvature is interpreted as the average distance that points on a sphere are moved by parallel transport when compared to the center of the sphere. 
I.e. a point $x$ has negative Ricci curvature if when we transport the sphere centered at $x$ to the sphere centered at a nearby point $y$, then on average points move further than $d(x, y)$. 
Similar notions work for positive and zero curvature. 

For arbitrary discrete spaces, Ollivier relies on the $L^1$ Wasserstein distance to define this curvature. 
This involves taking the infimum over functions defined on probability measures on the spheres about $x$ and $y$. 
In \cite{medscale}, this notion is studied for Cayley graphs. 
However, in the setting of group theory there is a natural correspondence between the sphere centered at $x$ and the sphere centered at $y$. 
Namely, we can translate by the group action, bypassing the need to take an infimum.

Let $G$ be a group with generating set $S$, and denote the Cayley graph with respect to $S$ as $\Gamma(G,S)$.
For any $x\in G$, let $|x|$ denote its length in the $(G,S)$ word metric.
Write $B_r(x)=\{g\in G\mid |x^{-1}g|\leq r\}$ and $S_r(x)=\{g\in G\mid |x^{-1}g|=r\}$ for the ball and sphere of radius $r$ centered at $x$.
\begin{definition}
We define $\mathcal{S}_r(x,y)$, the {\em spherical comparison distance of radius $r$} to be: $$\mathcal{S}_r(x,y):=\frac{1}{|S_r|}\sum_{w\in S_r}d(xw,yw).$$

Similarly, we define the {\em comparison distance of a ball of radius $r$} to be: $$\mathcal{B}_r(x,y):= \frac{1}{|B_r|}\sum_{w\in B_r}d(xw,yw).$$ 
$\mathcal{S}_r(x,y)$ (resp. $\mathcal{B}_r(x,y)$) is the average distance between $xw$ and $yw$, over all $w\in S_r$ (resp. $B_r$).
\end{definition}

If on average translating points $x,y$ by $w\in S_r$ moves them closer together than they were before, $\mathcal{S}_r(x,y) < d(x,y)$, while if it on average moves them further away, then $\mathcal{S}_r(x,y)>d(x,y)$ (similar for $\mathcal{B}_r(x,y)$).
Using these quantities, a new Ricci curvature can be defined via 
$$\kappa^\mathcal{B}_r(x,y)=\frac{d(x,y)-\mathcal{B}_r(x,y)}{d(x,y)} \mbox{ or } \kappa^\mathcal{S}_r(x,y)=\frac{d(x,y)-\mathcal{S}_r(x,y)}{d(x,y)},$$ 
which is called the \textit{comparison curvature}.
Positive comparison curvature corresponds to the case where $\mathcal{S}_r(x,y) < d(x,y)$ while negative comparison curvature corresponds to the case where $\mathcal{S}_r(x,y)>d(x,y)$.

In \cite{medscale} it was observed that these quantities can be rewritten in a more computable form: By symmetry, $\kappa_r^\mathcal{B}(x,y)=\kappa_r^\mathcal{B}(e,x^{-1}y)$, so $d(xw,yw)$ can be rewritten as $|w^{-1}x^{-1}yw|$ (when $e$ is used as a comparison point, $\kappa_r^\mathcal{B}(e,x^{-1}y)$ is written $\kappa_r^\mathcal{B}(x^{-1}y)$).
They also restricted their attention mostly to the case where $r=1$, in which case $\kappa_1^\mathcal{B}(g)$ was rewritten $\kappa(g)$.
In the case of $r=1$, $\mathcal{S}_1(g)=\mathcal{B}_1(g)$ is given by the average word length of the conjugate of $g$ by the generating set, which was abbreviated to $\GenCon(g)=\frac{1}{|S|}\sum_{a\in S}|a^{-1}ga|$.
Thus, the notion of Ricci curvature studied in \cite{medscale} was ultimately defined to be $$\kappa(g)=\frac{|g|-\GenCon(g)}{|g|}.$$

It is a rather delicate property–for instance, it relies on a specific choice of generating set.
The authors of \cite{medscale} were able to show a number of properties of $\kappa(g)$, such as characterizing uniform zero and positive curvature.
More recently, it has been shown that uniform negative curvature implies exponential growth \cite{expgrowth}.
It was noted that positive curvature seemed in some sense rare in many groups.
Indeed, it is fairly easy to guess which elements have zero or negative curvature in natural generating sets, as one can look to central elements for zero curvature, while any kind of ``free-ness'' tends to yield negative curvature.
In contrast, finding positive curvature often required more contrived generating sets, and there were no known usual suspects that could be expected to yield positive curvature.
Another observation was that restricting $\kappa(g)$ to radius 1 was in some sense an arbitrary choice, and that it would be interesting to study properties of this curvature at larger scales.
Tying this into the previous question, it was observed that some instances of positive curvature disappeared upon increasing this radius, leading one to speculate whether positive curvature is possibly an artefact of small radius.

In this note, we address these questions, and find examples of positive curvature that are stable under increasing comparison radius.
Notably, we find a class of elements that seem to readily yield non-negative curvature with respect to standard generating sets, namely dead-end elements.

\begin{definition}
Let $(G,S)$ be a group $G$ with generating set $S$.
Then a \textit{dead-end element} $g\in (G,S)$ is an element such that no geodesic ray from the identity to $g$ can be extended past $g$.

Furthermore, $g$ has \textit{depth $k$} if the minimum length of a path from $g$ to an element strictly further away from the origin is of length $k$.

We say $g$ has \textit{strict depth} $k$ if $|g|>|ga_1|>|ga_1a_2|>\dots >|ga_1a_2\dots a_k|$ for any $a_1,\dots ,a_k\in S$.
\end{definition}

Dead-end elements seem to be a useful place to look for non-negative curvature. Namely we have the following:

\begin{proposition}
Let $(G,S)$ be a group $G$ with generating set $S$.
 Let $g\in G$ be a strict dead end element of depth $k\geq 1$.
 Then $\kappa_r^S(g)\geq 0$ for all $r<k$.
\end{proposition}
\begin{proof}
	Consider the quantity $d(w, gw)$ for $w\in S_r$. 
	Since $g$ has strict depth $k>r$ we see that $|gw|\leq |g|-r$ and so we see that $d(gw, w)\leq |g|$.
	Thus we see that the curvature of $g$ is $\geq 0$. 
\end{proof}

We study such elements in the Lamplighter group $\lamptwo$ and Houghton's group $H_2$. 
In both these cases we find that the dead-end elements have 0 curvature. 
We find positive curvature from related elements known as backtrack elements. 

\begin{definition}
Let $G$ be a group, and let $g\in G$ be a dead end element of depth $k$.
Then $w$ is a {\em backtrack element} if $w=gw'$ for some $w'\in G$ of length less than $k$, and $|w|\leq|g|$.
\end{definition}

In Sections 2, we show that certain backtrack elements yield positive curvature $\kappa_r$ in the Lamplighter group $\lamptwo$. 

\begin{customthm}{\cref{thm:lllargeradii}}
Let $\integers_2 \wr \integers$ be the Lamplighter group $\lamptwo$ with the standard generating set $\{a,t\}$.
 Then for arbitrary choice of $r$ there are backtrack elements $g$ such that $\kappa_r^S(g)>0$.
\end{customthm}

In section 3 we prove the analogous theorem for Houghton's group $H_2$. 

\begin{customthm}{\cref{houghton}}
	Let $H_2$ be Houghton's group, with the standard generating set $\{s, \sigma\}$a.
	Then for arbitrary choice of $r$, there are backtrack elements $g$ such that $\kappa_r^S(g)>0$.
\end{customthm}

In \cite{medscale}, it was shown that the discrete Heisenberg group $H(\integers)$ has a positive density of elements with all three signs of curvature $\kappa_1^S$.
 In Section 4, we generalize this result and show that one still has positive density of all three signs in $H(\integers)$ with respect to an arbitrary comparison radius:

 \begin{customthm}{\cref{heisenberg}}
There is a positive density of elements $g\in H(\integers)$ satisfying $\kappa_r^\mathcal{S}(g)>0$, $\kappa_r^\mathcal{S}(g)=0$ and $\kappa_r^\mathcal{S}(g)<0$ for arbitrary $r$.
\end{customthm}

The authors thank Jennifer Taback for helpful conversations on the project.

\section{Positively curved elements in the lamplighter group}

In this section, we will give elements of the lamplighter group $\Z_2\wr\Z$, which have positive medium scale curvature.
We will show that these elements can be taken so that the comparison radius is arbitrarily large. 

Recall, the lamplighter group $\lamptwo$ is the semidirect product $\mathcal{L}_2 = (\oplus_{i\in\Z}\Z_2)\rtimes \Z$.
Here $\Z$ acts by permuting the factors. This group is generated by two elements $a, t$ where $a$ is a generator of one of the copies of $\Z_2$ and $t$ is a generator of $\Z$. 

This group is often called the Lamplighter group $\lamptwo$, due to the following interpretation: Consider a street extending infinitely in both directions.
On this street there is a set of lamps, indexed by the integers, which can either be on or off.
A lamplighter starts at the lamp indexed by zero, and can either move one position to the right (apply $t$), one position to the left (apply $t^{-1}$) or switch the state of the lamp he is currently standing at (apply $a$).
The elements of the lamplighter group $\lamptwo$ are given by configurations of the lamplighter's final position and only finitely many lamps turned on.
We will refer to this as the ``lamplighter picture.''

In \cite{taback}, it was shown that any element of the lamplighter group has a geodesic spelling given as follows: Let $a_i=t^{-i}at^i$.
Suppose that the element $g\in \lamptwo$ is the configuration with lit lamps at $\{-j_k,-j_{k-1},\dots ,-j_1,i_1,i_2,\dots ,i_l\}$ (where $j_s$ is positive, and the set of indices is ordered from lowest to highest value) and lamplighter's final position is at lamp $m$.
If $m\geq 0$, then a geodesic spelling for $g$ is given by $a_{-j_k}a_{-j_{k-1}},\dots ,a_{-j_1}a_{i_1}\dots a_{i_l}t^m$. 
If $m<0$ then the geodesic spelling is given by $a_{i_l}\dots a_{i_1}a_{-j_1}\dots a_{-j_k}t^m$.
In \cite{taback}, an explicit formula for word length is also obtained: Denote the length of an element $g\in\lamptwo$ to be $D(g)$.
Then $$D(g)=k+l+min\{2j_l+i_k+|m-i_k|,2i_k+j_l+|m+j_l|\},$$ where $k$ is the number of lit lamps with positive index, $l$ is the number of lit lamps with negative index, $i_k$ is the index of the rightmost lit lamp, $j_l$ is the index of the leftmost lit lamp and $m$ is the final position of the lamplighter.

We now look for the elements of positive curvature.
The elements come from the dead elements $d_m$ from \cite{taback}, where $d_m = a_0a_1\dots a_ma_{-1}\dots a_{-m}$.
The dead end depth of the element $d_m$ is $m$. 
\begin{example}
$d_3=a_0a_1a_2a_3a_{-1}a_{-2}a_{-3}$ is a dead end element of depth 3 with length 19.
The following set of elements is a path from $d_3$ to an element of length 20: $d_3, d_3t, d_3t^2,d_3t^3, d_3t^4, d_3t^5, d_3t^6, d_3t^7$.
The lengths of the elements in this path are: 19, 18, 17, 16, 17, 18, 19, 20.
\end{example}

In the ``lamplighter picture,'' an element $d_m$ corresponds to the lamplighter lighting up all lamps from $[-m,m]\subset \integers$ and then returning to the origin.
Elements of the form $d_mt^i$ can be viewed as the lamplighter's ``escape path'' from the dead end element $d_m$, with length decreasing until $|i|\geq m$.
Notice that these are examples of backtrack elements from the dead end word $d_m$.
For any $d_m$, there are $2m-1$ such elements of this form.

\begin{proposition}
The set of elements $\{d_mt^k\}_{0<|k|<m}$ have positive curvature for all $m$.
\end{proposition}
\begin{proof}
	Consider an element $g=d_mt^k$, $0<k<m$.
	In order to calculate the curvature $\kappa(g)=\frac{|g|-GenCon(g)}{|g|}$ we will first need to determine $D(g), D(t^{-1}gt), D(tgt^{-1}), D(aga)$.

	For the first note that conjugation by $t$ or $t^{-1}$ has the property of moving the lit lamps one to the left or right.
	By symmetry we can focus on the case that all the lamps are lit in the interval $[-m-1, m-1]$ and the lamplighter ends at position $k$. 
	The length of such an element is $6m-k+1$ by the formula above. 
	We can also see that $D(g) = 6m-k+1$. 

	For the length of $aga$, we can see that we have changed the set of lit lamps. 
	We can see that the set of lamps in $[-m, m]$ are lit with the exceptions of lamps at position 0 and $k$. 
	Thus, from the formula above, we see that $D(aga) = 6m-k-1$. 

	Thus we can conclude that $\kappa(g)>0$. 
\end{proof}

The ``defect'' that results in positive curvature is coming from conjugation by $a$ switching the lights at position 0 and position $k$.  

\subsection{Other wreath products}

Using this observation we can create the same phenomenon in other wreath products
Consider $\mathcal{L}_n=\integers_n\wr \integers$.
This is the lamplighter group of rank $n$, and can be understood easily by generalizing the lamplighter picture of $\lamptwo$, where now instead of lamps with two settings ``on'' and ``off'' we have n possible settings (say, $n$ different colors a lamp can take).
The word length formula for $\lamptwo$ only depends on the indices of lit lamps and the final position for the lamplighter, so it is easy to see that it generalizes directly to $\mathcal{L}_n$ for the generating set $\{s_1,s_2,\dots ,s_{n-1},t\}$ where $s_i=i\in\integers_n$.
From here, it is also not hard to see that the same is true for $A\wr \integers$, for any finite group $A$.
We can use the generating set $A\smallsetminus\{e\}\cup\{t\}$, where now the lamps change colors based on the multiplication in $A$.
Since the word length formula is again essentially unchanged, we can define analogues of $d_m$ in $A\wr \integers$ by replacing instances of $a_i=t^iat^{-i}$ with $s_{i,k}=t^i s_k t^{-i}$ for any choices of $s_k\in A\smallsetminus \{e\}$, to get dead end elements of depth at least $m$.

Just as before, we can backtrack from the $d_m$ analogues in $A\wr \integers$.
There will still be ``defects'' in the form of unlit lamps in $[-m,m]$.
The main difference between this situation and the case with $\lamptwo$ is that now some care must be taken when determining which elements turn lamps off.
In particular, conjugating $g=d_mt^k$ by $a$ in $\lamptwo$ is guaranteed to turn off lamps at the origin and at index $k$, while now it is possible that conjugating $g$ by $s_i$ will simply switch the lamps at the origin and at index $k$ to a different color, thus not decreasing word length.
However, since every possible color has length 1 in this generating set, in these cases word length will remain unchanged under conjugation, except in the cases where $s_i$ is the inverse of the color at the origin or the color at index $k$.
This means that conjugating $g$ by $w\in S$ at least keeps word length the same, and some cases are guaranteed to strictly lower the word length of $g$.
Thus we have proved the following:
\begin{corollary}
Let $A$ be a finite group.
Consider $A\wr \integers$ with generating set $$S = A\smallsetminus\{e\}\cup\{t\}.$$
Let $s_{i,k}=t^is_kt^{-i}$.
Then for any $m$ and $l$ such that $|l|<m$, elements of the form $s_{0,k_0}s_{1,k_1}\dots s_{m,k_m}s_{-1,k_{-1}}\dots s_{-m,k_{-m}}t^l$ have positive curvature for any choices of $s_{j,k_j}$.
\end{corollary}

We now come back to the case of the lamplighter group $\mathcal{L}_2$. 
We will show that the positive curvature exhibited earlier is persistent under taking larger radii. 
Previously we conjugated by generators, however, for curvature over a sphere of radius $r$, we must instead conjugate by all elements of length $r$.
Let us warm up with an example for radius 2. 

\begin{example} 
	We will now consider $\kappa_2^\mathcal{S}(d_mt^k)$.
	Note where before we only required that $k<m$ (i.e. $|m-k|\geq 1$), we now require $|m-k|\geq 2$.
	There are six elements of length two in $\lamptwo$: $at$, $ta$, $t^2$ and their inverses.
	By calculation, we have the following: 
	$$D(at g t^{-1} a) = D(ta g at^{-1}) =6m-k-3$$ 
	$$D(at^{-1}gta)=D(t^{-1}agat)=6m-k-3$$ 
	$$D(t^{-2}gt^{2})=D(t^{2}gt^{-2})=D(g)=6m-k+1.$$
	Then $\frac{1}{|\mathcal{S}_2|}\sum_{w\in \mathcal{S}_2}|w^{-1}gw|=6m-k-\frac{4}{3}$, so $\kappa_2^\mathcal{S}(d_mt^k)=\frac{5}{2|d_mt^k|}>0$.
\end{example}

We now show how to extend this to get elements of positive curvature for arbitrary radius. 
These elements will be of the form $d_mt^k$ for appropriately chosen $k, m$. 
We begin with the following lemma.

\begin{lemma}\label{llconjtr}
	Suppose that $r, k>0$ and $m>k+r$. Then $$D(t^{-r}d_mt^{k+r})=D(t^rd_mt^{k-r})=D(d_mt^k).$$
\end{lemma}
\begin{proof}
	The word length formula for $\lamptwo$ depends only on the lamps which are lit and the final position of the lamplighter.
	Considering the case of $t^{-r}d_mt^{k+r}$ we see that all the lamps in the interval $[-m+r, m+r]$ are lit and the lamplighter ends at position $k$. 
	When we substitute this into the formula for $D(g)$ we arrive at $6m-k+1$ which is exactly $D(d_mt^k)$.
	
	In the other case the set of lit lamps are all lamps in the interval $[-m-r, m-r]$ and the lamplighter also ends at position $k$. 
\end{proof}

We can in fact strengthen the previous lemma as follows.
\begin{lemma}\label{llconjgen}
	Suppose that $w$ is an element of the lamplighter group such that the lamps $-m, m$ are lit and the lamplighter is at position $k\in (-m, m)$. 
	If $m>|k|+|r|$, then $D(t^rwt^{-r}) = D(w)$. 
\end{lemma}
\begin{proof}
	The proof is similar to the previous lemma.
	The word length formula relies on 4 quantities: number of lamps lit, left most lit lamp, right most lit lamp, position of lamplighter. 
	In the case $k>0, r>0$, these quantities are given by $l_w, -m+r, m+r, k$ where $l_w$ is the number of lamps lit for the element $w$. 
	Substituting these into the word length formula and using the fact that $m>k+r$ we arrive at the desired conclusion. 
\end{proof}

We are now ready to prove the main theorem of this section. 

\begin{theorem}\label{thm:lllargeradii}
	Let $r\in \N$. Let $k>0$. If $m>r+k$, then $\kappa^r(d_mt^k)>0$. 
\end{theorem}
\begin{proof}
	We have already seen in Lemma \ref{llconjtr} that conjugating by $t^r$ leaves word length invariant. 

	We now consider a general element $g$ of the lamplighter group such that $D(g)= r$. 
	We can assume that $g$ switches at least one lamp. 
	Since $D(g)=r$, we know that $g$ cannot switch any lamps outside $[-r, r]$.
	Suppose that $g$ switches lamps at positions $i_1, \dots, i_l$ and ends at position $p$. 
	Then $g = a_{i_1}\dots a_{i_l}t^p$. 
	We have already noted that $l<r$. 

	Now consider $$g^{-1}d_mt^kg = t^{-p}a_{i_l}\dots a_{i_1}d_mt^ka_{i_1}\dots a_{i_l}t^p.$$ 
	We will show that $D(a_{i_l}\dots a_{i_1}d_mt^ka_{i_1}\dots a_{i_l})< D(d_mt^k)$.
	The element $$h = a_{i_l}\dots a_{i_1}d_mt^ka_{i_1}\dots a_{i_l}$$ is $d_mt^k$ except we have switched the lamps at position $a_{i_1}, \dots, a_{i_l}$ and $a_{i_1-k}, \dots, a_{i_l-k}.$
	This may result in some being switched twice however the lamps at position $i_1-k$ and $i_l$ are both switched off. 
	Since $m>k+r$, we see that the element $h$ corresponds to a configuration where the lamps $-m, m$ are both lit and the lamp lighter is at position $k$. 
	Since we still have the same extremal lamps lit and the same position of the lamplighter we have shortened the length since we are lighting less lamps.
	I.e. $D(h) < D(d_mt^k)$. 

	Since we have not changed the state of the lamps at position $-m, m$, we can appeal to Lemma \ref{llconjgen} to compute $D(g^{-1}d_mt^kg) = D(t^{-p}ht^p) = D(h) < D(d_mt^k)$. 
	Thus we have shown that conjugation by elements in the sphere of radius $r$ either shorten $d_mt^k$ or leave the length unchanged. 
	We conclude that $\kappa^r(d_mt^k)>0$. 
\end{proof}

Since the above proof works for $r<m-k$, we see that the element $d_mt^k$ has positive curvature when we consider using balls instead of spheres.

We end this section with an elementary yet interesting observation about $\lamptwo$:
\begin{proposition}
Let $w$ be a word in $\lamptwo$.
Then there exists an $M\in \mathbb{N}$ such that for all $m>M$, $w$ is a subword in $d_m$.
\end{proposition}
\begin{proof}
Let $w=a_{i_1}a_{i_2}\dots a_{i_l}a_{j_1}\dots a_{j_k}t^s$.
Notice that each of the $a_{x_y}$ commute with each other.
Suppose $s\neq x_y$ for any $x_y\in \{i_1,i_2,\dots ,i_l,j_1,\dots ,j_k\}$.
Then $wat^{-s}=a_{i_1}a_{i_2}\dots a_{i_l}a_{j_1}\dots a_{j_k}a_s$.
Let $M=$max$\{|i_l|,|j_k|,|s|\}$.
Then we can multiply $wat^{-s}$ by $a_q$ for all $q\in [-M,M]\setminus \{i_1,\dots ,i_l,j_1,\dots ,j_k,s\}$ to get the dead end word $d_M$, and by commutativity of the $a_{x_y}$'s we can arrange them into a geodesic spelling of $d_M$.
From here it is obvious that we can produce higher length dead end words $d_m$.
If $s= x_y$ for some $x_y\in \{i_1,i_2,\dots ,i_l,j_1,\dots ,j_k\}$, then we multiply $w$ by $t^{o}at^{-(o+s)}$, where $o$ is the minimum number of steps from $s$ to an unlit lamp.
From here we repeat the above process to produce $d_M$.
\end{proof}

\section{Positively curved elements in Houghton's Group}

In this section, we demonstrate elements of Houghton's Group $\mathcal{H}_2$ that exhibit positive curvature.
Similar to the $\lamptwo$ case, these elements are ``back-track'' elements from dead end words in $\mathcal{H}_2$.
Let $*^n$ be the disjoint union of $n$ copies of $\mathbb{N}$.
Houghton's group $\mathcal{H}_n$ can be understood as a group of permutations $*^n$ such that the permutations act as translations outside of a finite subset of $*^n$.

We will consider $\mathcal{H}_2$, which in the style of \cite{lehnert} we will view via the following geometric interpretation: 
Consider a biinfinite string of beads, indexed by nonzero integers, and a ``lamplighter'' who stands between two beads.
The identity word represents all beads in their original order and the lamplighter standing between $-1$ and $1$.
The generator $s$ moves the lamplighter one position to the right, and $\sigma$ transposes the beads to the immediate left and right of the lamplighter.
We can also view this group as $S_{\infty}\rtimes\Z$. 
Here $S_{\infty}$ is the set of permutations with finite support and $\Z$ acts by translation. 
One should make the appropriate adjustment for the fact that $0$ is not in the set we are permuting.
See \cite{lehnert} for further discussion as well as an explicit presentation for $\mathcal{H}_2$.

In \cite{lehnert}, a set of dead end elements $g_k$ is shown to exist.
The element $g_k$ transposes the beads labeled $l$ and $-l$ for $1 \leq l \leq k$ and returns the lamplighter to the origin.
We will consdier the elements $h_{k, m}$ which is the element that transposes the beads labelled $l$ and $-l$ for $m\leq l\leq k$ and returns the lamplighter to the origin. 

To show that $g_k$ is a dead end element, the following geodesic spelling was defined: Let $u_l=s^{-(l-1)}(\sigma s)^{2(l-1)}(\sigma s^{-1})^{2(l-1)}\sigma s^{l-1}$.
The four terms in this spelling can be viewed as: travelling to $-l$ ($s^{-(l-1)}$), making swaps and moving right until we arrive at $l$ ($(\sigma s)^{2(l-1)}$), then making swaps and moving left until we arrive at $-l$ again ($(\sigma s^{-1})^{2(l-1)}$) and finally returning to the origin $(s^{l-1}$). 
One should note that moving left first was an arbitrary choice and a geodesic spelling can also be obtained by moving right first. 
This spelling reverses the four terms in the previous spelling. 
We will refer to these spellings as negfirst and posfirst and denote them $u_l^-$ and $u_l^+$ if needed. 
We will use $u_l$ when it does not matter which spelling is used. 

We can obtain a geodesic word for $h_{k, m}$ by concatenating the $u_l$ namely, $$h_{k, m} = u_k\dots u_m = u_m\dots u_k.$$ 
Where one uses either all posfirst or all negfirst. 
Thus, we have 4 possible geodesic spellings for $h_{k, m}$. 

\begin{proposition}
	For any $1<k\leq m$, we have $\kappa(h_{k, m}) > 0$. 
\end{proposition}
\begin{proof}
	Since $h_{k, m}$ is a permutation of $\Z\smallsetminus\{0\}$ and has disjoint support from $\sigma$, we see that these two elements commute. 
	Thus conjugation by $\sigma$ does not change word length. 

	Now let us consider conjugation by $s$. 
	By using the posfirst spelling, the geodesic representative for $h_{k, m}$ starts with an $s$ and ends with an $s^{-1}$. 
	Thus we can see that $s^{-1}h_{k,m}s$ cancels out the first and last $s$ and thus has shorter length. 
	Similarly using the negfirst spelling for $s^{-1}$ we arrive at the same conclusion. 
	Thus we have two elements that shorten length and one that leaves it invariant.
	Thus $\kappa(h_{k, m})>0$. 
\end{proof}

In fact, these elements have positive curvature for arbitrary radius.
To see this we have the following elementary lemma.

\begin{lemma}\label{lem:shwshc}
	Let $g$ be an element of $H_2$. If $g$ moves the element $r\in\Z\smallsetminus\{0\}$, then $|g|\geq r$.
\end{lemma}
\begin{proof}
	Suppose that $r>0$.
	To permute $r$ to another position, the ``lamplighter'' must first move to position $r-1$ which requires $r-1$ instances of $s$. 
	We must also have at least one instance of $\sigma$. 
	Thus $|g|\geq r$. 
\end{proof}

\begin{theorem}\label{houghton}
	For $1<k\leq m$ and $r<k$, we have $\kappa^r(h_{k, m})>0.$
\end{theorem}
\begin{proof}
	We will show that conjugating by an element in the sphere of radius $r$ either leaves the length of $h_{k, m}$ invariant or shortens it. 
	
	In the semi-direct product structure, we can see that $h_{k, m}$ is an element of $S_{\infty}$.
	Let $g$ be an element in the sphere of radius $r$. 
	If $g$ is an element of $S_{\infty}$ we can see from Lemma \ref{lem:shwshc} that the support of $g$ is disjoint from the support of $h_{k, m}$. 
	Thus $g^{-1}h_{k, m}g = h_{k, m}$ and the length is unchanged. 

	If $g\notin S_{\infty}$, then we can write it in the form $\tau s^i$ for some $i\neq 0$. 
	Since $|g| = r$, we deduce that $i\leq r$. 
	Lemma \ref{lem:shwshc} shows that the support of $\tau$ and the support of $h_{k, m}$ are disjoint. 
	Thus, we obtian $|g^{-1}h_{k, m}g| = |s^{-i}h_{k, m}s^i|$. 
	
	We will examine the case $i>0$. 
	By spelling $h_{k, m}$ posfirst, we see that the spelling starts with $m$ instances of $s$ and ends with $k$ instances of $s^{-1}$. 
	Since $i\leq r\leq k$ we see that this conjugation cancels both the first $i$ letters and the last $i$ letters of $h_{k, m}$. 
	Thus we have reduced the length of the word by at least $2i$. 

	We have seen that every element in the sphere of radius $r$ leaves the element unchanged or shortens its length. 
	Thus we conclude that $\kappa^r(h_{k, m})>0$. 
\end{proof}

The dead end elements in both $\lamptwo$ and $\mathcal{H}_2$ can be informally understood as elements that ``highly mix'' the original configurations of the street/string of beads.
In $\lamptwo$, this looks like turning on all lamps in a symmetric interval centered at 0.
In $\mathcal{H}_2$ this looks like moving each bead in a symmetric interval centered at zero to it's opposite sign index.
Multiplying by certain small words changes the configuration in the symmetric interval, and since the dead end elements highly mix the original configuration this inevitably results in a configuration that is closer to the original than before.
This can be understood in analogy with highest height elements in finite Coxeter groups.
The positive curvature appears in much the same way, where as long as conjugating the ``backtrack'' element stays within the highly mixed interval, we would expect the conjugated word to have length less than the original word.

\section{Increasing the spherical comparison radius for the Heisenberg group}

Let $H(\integers)$ denote the discrete Heisenberg Group.
In \cite{medscale} it was shown that $H(\integers)$ has a positive density of elements with all three signs of $\kappa_1^\mathcal{S}$.
In this section we generalize this result to show positive density of all three signs with respect to arbitrary comparison radius.

$H(\integers)$ has a standard presentation $\big < a,b :[a,b]$ is central $\big >$.
In the style of \cite{medscale}, we use Mal'cev coordinates to spell words in $H(\integers)$: Every group element can be uniquely written in the form $a^Ab^Bc^C$ where $A,B,C\in \integers$ and $c=[a,b]$.
The word metric obeys the following formula: If $A>B>0$ and $C>0$ we have 
$$|a^Ab^Bc^C|=\bigg \{ \begin{matrix} 2\lceil C/A \rceil +A +B, & C\leq A^2-AB \\ 2\lceil 2 \sqrt{C+AB}\rceil -A - B, & C\geq A^2-AB \end{matrix}$$ 
We call the first case $(C\leq A^2-AB)$ the low height case and the second $(C\geq A^2-AB)$ the high height case.
In \cite{medscale} it was shown that $H(\integers)$ with the standard generating set has a positive proportion of points with $\kappa(g)>0, \kappa(g)<0$ and $\kappa(g)=0$.
We expand on this showing that the result holds true for arbitrarily large spherical comparison radius.

\begin{theorem}\label{heisenberg}
	There is a positive density of elements $g\in H(\Z)$ satisfying $\kappa_r^X(g)>0, \kappa_r^S(g) = 0$ and $\kappa_r^S(g)<0$ for arbitrary $r$. 
\end{theorem}

\begin{proof}
	When the spherical comparison radius was kept at 1, it was only required that $A>B>0$ and $C>0$.
	As we increase the spherical comparison radius to $r>1$, we will now require additional restrictions:
	We will be considering the curvature of elements in a subset $U\subseteq B_n$ for some $n>2r$.
	We will study elements in the annulus $B_k\smallsetminus B_{2r}$, where $k$ is chosen significantly larger than $2r$ and also $k<n-4r$. 
	This is so that conjugating something in the annulus remains in $B_n$. 
	Define $U$ to be the subset of $B_k\setminus B_{2s}$ satisfying:
	\begin{enumerate}  
		\item $C-As\geq 0$
		\item $A-B\geq 2s$
	\end{enumerate} 
	Condition $(1)$ is to ensure that conjugating by $a^s$ (and $b^s$) does not lower the value of $C$ below 0, which would invalidate our word length formula.
	Condition $(2)$ is akin to saying that $A$ and $B$ are sufficiently larger than $s$.
	Of course, adding these conditions runs the risk of restricting our set to one that is too small to make positive measure arguments, but it will be shown later that $U$ is large enough for our purposes as long as $k$ is significantly larger than $s$.

	Let $a^Ab^Bc^C\in U$.
	Let $w=a^{A'}b^{B'}c^{C'}\in B_s$, and consider $|w^{-1}a^Ab^Bc^C w|$.
	First note that since $c$ is central, any instance of $c$ in $w$ will cancel itself out upon conjugation, and will thus not change word length.
	Next, notice by the observation made in \cite{medscale} that conjugating by $b^{\pm}$ changes $C$ by adding and subtracting $A$, and so conjugation by $b^\pm$ has no effect on curvature (notice that since $a^Ab^Bc^C$ is in $B_n\setminus B_{2r}$ we can subtract $2r$ from the wordlength without issue).
	Thus we must only consider how many occurrences of $a$ there are in $w$.
	Let there be $t\leq r$ occurrences of $a$ in $w$ (implicitly we set $t$ to be the net sign of the occurences of $a$ and $a^{-1}$, which without loss of generality we assume is positive).
	As noted in \cite{medscale}, since $A$ and $B$ are kept constant, the change in word length upon conjugation by $a$ or $a^{-1}$ is dependent on the quantities $C-B$ mod $A$ and $C+B$ mod $A$ respectively.
	Precisely, for $C=kA+s$ we have: 
	$$\lceil (C+B)/A\rceil = \bigg\{
		\begin{matrix} k+2, \; s > A-B \\ k+1, \; 
			s\leq A-B 
		\end{matrix}$$ 
			$$\lceil (C-B)/A \rceil = \bigg \{ 
				\begin{matrix} k+1, \; s>B \\ k, \; s \leq B \end{matrix}$$

	Now, since there are $t$ occurrences of $a$ in $w$ (and $t$ occurrences of $a^{-1}$ in $w^{-1}$), we replace $C+B$ with $C+Bt$ and $C-B$ with $C-Bt$.
	Thus, our new formulas are: 
	$$\lceil (C+Bt)/A\rceil = \bigg\{
		\begin{matrix} k+2, \; s > A-Bt \\ k+1, \; s\leq A-Bt \end{matrix}$$ 
		$$\lceil (C-Bt)/A \rceil = \bigg \{ 
			\begin{matrix} k+1, \; s>Bt \\ k, \; s \leq Bt \end{matrix}$$

	Let $r>0$, so we are considering elements from the annulus $B_k\setminus B_{2r}$, with $k>>2r$.
	First note that low height elements with positive density in $B_k\setminus B_{2s}$ have positive density in $B_n$: The measure of low height elements in $B_k\setminus B_{2r}$ is given by 
	$$\int_0^k \int_0^{k-A}(A^2-AB)dAdB-\int_0^{2r}\int_0^{2r-A}(A^2-AB)dAdB=\frac{k^4}{24}-\frac{16r^4}{24}.$$
  Since $k>>r$ this quantity is of the order $O(k^4)$, which has positive measure in $B_k$ since $|B_k|$ is also of this order.
	It is also clear that $B_k\setminus B_{2r}$ has positive density in $B_n$, so we can even say that low height elements in our annulus take up a positive proportion of the original ball $B_n$.
	Elements in the sphere of radius $r$ in $H(\integers)$ can be written $a^tb^{B'}c^{C'}$ where $|a^tb^{B'}c^{C'}|\leq r$ and $0\leq t\leq r$.
	For each $t$, we divide the possible values $s$ can take into three cases: $1\leq s\leq Bt$ (Case $X_t$), $Bt\leq s \leq A-Bt$ (Case $Y_t$) and $A-Bt\leq s\leq A-1$ (Case $Z_t$).
	If a subset of remainders satisfies $X_t$,$Y_t$ or $Z_t$ for all $t\leq r$, say that the subset satisfies $X$, $Y$ or $Z$ respectively.
	Consider the sector satisfying $\frac{1}{5r}A\leq B \leq \frac{2}{5r}A$.
	Then at least $\frac{1}{5r}$ of the possible remainders satisfy Case $X$ (in particular, at least $\frac{1}{5r}$ of the remainders satisfying Case $X_1$, and hence satisfy Case $X_t$ for all $t\geq 1$).
	Similarly, at least $\frac{1}{5r}$ satisfies Case $Y$ and at least $\frac{1}{5r}$ satisfies Case $Z$.
	For elements satisfying Case $X$, $\lceil \frac{C\pm Bt}{A}\rceil=k+1,k$, so $\kappa_r^\mathcal{S}(g)>0$.
	For elements satisfying Case $Y$, $\lceil\frac{C\pm Bt}{A}\rceil=k+1,k+1$ so $\kappa_r^\mathcal{S}(g)=0$.
	For elements satisfying Case $Z$, $\lceil\frac{C\pm Bt}{A}\rceil=k+2,k+1$ so $\kappa_r^\mathcal{S}(g)<0$.
	Thus we have shown that there is a positive proportion set in $B_k$ where the sign of curvature repeats periodically mod $A$ in the low height case for any spherical comparison radius $r$.

	Of course, periodicity in the curvature values is only useful if it can be shown that the set $U\cap \{$low height elements$\}$ contains ``many'' full periods mod $A$.
	First we must check that $A^2-AB\geq A \geq B+2r$ for most choices of $A$, where the first inequality establishes the low height case and where the second inequality is Condition (2) defining $U$.
	Note that for $A\geq 5r$ (and with our bounds on $B$ in place), we have $A^2-AB\geq A \geq B+2r$ to be equivalent to $25r^2-10r\geq 5r \geq 2+2r$, which is true for all $r\geq 1$.
	Since $k>>r$, most words in our ball satisfy $A\geq 5r$.
	Also note that $A\geq 5r$ will always result in a word of length greater than $2r$, which agrees with our original restriction to the annulus $B_k\setminus B_{2r}$.

	To fully restrict to $U$ we must impose the additional condition $C-Ar\geq 0$, and show that this still results in a set with positive density.
	Using the low-height word length formula, we find the following bounds for possible values of $C$ for words in $B_k$: Let $w\in B_k\setminus B_4$, so $|w|=l$ with $4\leq l\leq k$.
	Then for all $|A+B|\leq l$ we have $l = 2\lceil \frac{C}{A}\rceil + (A+B)$, so $\frac{l-(A+B)}{2}=\lceil\frac{C}{A}\rceil$.
	Then: $A(\frac{l-(A+B)}{2}-1)\leq C \leq A(\frac{l-(A+B)}{2}+1)$, or $A(\frac{l-(A+B)}{2})-A\leq C \leq A(\frac{l-(A+B)}{2})+A$.
	This means that for all values of $A$ and $B$ satisfying $A+B\leq l-2$ there are $2A$ values of $C$ in the low height case.
	Now we consider the second inequality.
	In particular we check what values satisfy both $C\geq Ar$ and the lower bound on $C$: $Ar\leq A(\frac{l-(A+B)}{2})-A \Rightarrow 2(r+1)\leq l-(A+B)$.
	Then for each choice of $A$ and $B$ with $A+B=l-2(r+1)$, there are $2A$ choices for $C$ such that $w=a^Ab^Bc^C$ is of length $l$.
	Once again appealing to the fact that $k>>r$, it is clear that a large proportion of elements in $B_k$ satisfy both $C\geq Ar$ and $A\geq 5r$.
\end{proof}

Note that elements satisfying Cases $X$ and $Z$ are not the only elements of positive/negative curvature in our region, just the ones where we are ``guaranteed'' those curvatures, or perhaps which elements take on the maximum pos/neg curvature.
It is reasonable to expect that if an element satisfies $X_t$ for many $t\leq r$ then it should still have positive curvature (respectively $Z_t$ and negative curvature).

Note that as the radius increases this method finds less and less zero curvature, as elements satisfying Case $Y$ should be the only ones with strictly zero curvature.
This is perhaps to be expected: As $H(\integers)$ is not a uniformly flat group, we can only hope for at most local flatness, and considering larger subsets should result in less strictly flat geometry.
In fact, it can probably be shown that if we take the radius to an appropriate limit the only elements with zero curvature should be the central elements $c=[a,b]^k$.

It should be noted that these are also the dead end elements in $H(\integers)$.
In \cite{alek}, it was shown that any geodesic word in $H(\integers)$ is the prefix of a dead end word.
The idea of the proof is to study the projection of $H(\integers)\rightarrow \integers^2$.
Under this projection, dead end words appear as oriented minimal perimeter polyominoe, with the area of the polyominoe corresponding to the $k$ in $[a,b]^k$.
Under this same projection, geodesic words appear as paths that weakly travel right and up, right and down, left and up or left and down in $\integers^2$.
The authors show that given any finite geodesic word, you can ``close'' the path represented by this word into an oriented minimal perimeter polyominoe, hence this geodesic word is the prefix of a dead end word as required.
The propensity of dead end words (and hence backtrack words) may help explain why we can consistently find large amounts of positive curvature.

\section{Relationship with Ollivier's Ricci Curvature}

In this section, we draw direct comparisons between comparison curvature $\kappa$ and Ollivier's notion of discrete Ricci curvature \cite{ollivier}.

\begin{definition}
	Let $(X,d)$ be a metric space and $\nu_1,\nu_2$ be two probabilty measures on $X$.
	Then the {\em set of transportation plans} $\Pi(\nu_1,\nu_2)$ is the set of measures on $X\times X$ projecting to $\nu_1$ on the first coordinate and $\nu_2$ on the second.
\end{definition}

\begin{definition}
	Let $(X,d)$ be a metric space and $\nu_1$, $\nu_2$ be two probability measures on $X$.
	The {\em $L^1$-transportation distance} between $\nu_1$ and $\nu_2$ is $$\mathcal{T}_1(\nu_1,\nu_2):=inf_{\zeta\in \Pi(\nu_1,\nu_2)}\int_{(x,y)\in X\times X}d(x,y)d\zeta(x,y).$$
\end{definition}

Using these definitions, we define Olliver's Ricci curvature $\kappa^*$ in our setting:

\begin{definition}
Let $(G,S)$ be a discrete group with finite generating set.
 Let $x,y\in G$.
 Then the {\em Ricci curvature in the direction $(x,y)$} is: $$\kappa^*(x,y)=1-\frac{\mathcal{T}_1(B_x,B_y)}{d(x,y)}$$ where $B_x$ and $B_y$ are the uniform distribution on unit balls at $x$ and $y$ respectively.
\end{definition}

Ollivier's original definition allows for $m_x$ to be any probability measures that vary measurably with $x$, but we have chosen this narrower definition to strengthen the parallels with $\kappa_1^S$.
Recall that $\kappa_1^S(x,y)=1-\frac{\GenCon(x,y)}{d(x,y)}$, where $\GenCon(x,y)$ is the average over the distance between $xs$ and $ys$ for all $s\in S$.

Notice that as each term $d(xs,ys)$ is given equal weight in $\GenCon(x,y)$, every point in the unit balls at $x$ and $y$ are given the same mass $\frac{1}{|S|}$.
Also notice that $\GenCon(x,y)$ chooses a very specific transportation plan between these distributions, namely the plan that takes the entire mass at $xs$ and sends it to $ys$.
However, Ollivier's transportation curvature uses the $L^1$ transportation distance between the two unit balls at $x$ and $y$.
This means that the analog of $\GenCon(x,y)$, namely $\mathcal{T}(B_x,B_y)$ is chosen as the infimum over all transportation plans.
This would allow, for instance, plans where $xs_1$ sends its mass to $ys_2$ for $s_1\neq s_2$ in $S$, or $xs$ sends a fraction of its mass to each point in $m_y$.
Therefore, it is immediately clear that $\kappa^*(x,y)\geq \kappa_1^S(x,y)$.

This raises obvious questions: When is the inequality strict? When are these quantities equal? 
 
These questions can be rephrased as follows: When does the transportation plan given by $\GenCon(x,y)$ give an optimal transport between the distributions $B_x$ and $B_y$?

In \cite{brezi} it was shown that given two sets of $n$ points $X$ and $Y$ and a cost function $c: X\times Y\rightarrow \mathbb{R}$, the optimal transport between $X$ and $Y$ minimizing $c$ is always given by a permutation.
As $B_x$ and $B_y$ are uniform distributions, we can treat them simply as sets of equal size, with cost function given by the word metric in $(G,S)$.
Thus, the above result shows that an optimal transport between $B_x$ and $B_y$ is in fact given by sending the entire mass at each point in $B_x$ to a point in $B_y$.
This means that means that $\mathcal{T}_1(m_x,m_y)$ can be determined by considering only finitely many transport plans, those given by the permutations on a set of size $|m_x|$.
Thus, it makes sense to study the set of permutations $S\to S$. 
One of these corresponds to the optimal transportation plan in $\Pi(m_x, m_y)$. 
Note that $\GenCon(x, y)$ corresponds to picking the identity permutation. 

We can now show that in certain cases, $\mathcal{T}_1(m_x,m_y)=\GenCon(x,y).$

\begin{example}
Consider $\integers^n$ with the standard generating set $S=\{a_1^{\pm 1},a_2^{\pm 1},\dot\dots,a_n^{\pm 1}\}$.
For any $x, y$, we have that $\kappa(x,y)=0$, so $\GenCon(x,y)=d(x,y)$. 
By the above discussion we can describe the data of $\zeta\in \Pi(m_x,m_y)$ with a bijection from $S$ to itself, which by a slight abuse of notation we also call $\zeta$.

Then the transport cost (ignoring the normalizing factor of $\frac{1}{|S|}$) for any $\zeta$ is given by:
\begin{align*}
	\sum_{a\in S}d(xa,y\zeta(a)) &= \sum_{a\in S}\|(x-y)+(a-\zeta(a)\|\\
				     &\geq \|\sum_{a\in S} (x-y) + (a-\zeta(a))\|\\
				     &= \|\sum_{a\in S} (x-y)\|\\
				     &= \sum_{a\in S} d(x, y).
\end{align*}
In the last line $a$ and $\zeta(a)$ cancel as $\zeta$ is a permutation. 

The above in fact shows that for any abelian group Ollivier's curvature agrees with ours. 
\end{example}

\begin{example}
	Consider $F_n=\langle a_1,a_2,\dots,a_n\rangle$, the free group on $n$ generators. 
	For any $x,y\in F_n$ we have the following, 
	$$\GenCon(x,y)=\frac{1}{|S|}\sum_{a_i\in S} d(xs,ys)=\frac{1}{|S|}\sum_{a_i\in S}|s^{-1}x^{-1}ys|=d(x,y)+2-\frac{2}{n}$$
	(see \cite{medscale} for details).
	
	Consider a transport plan $\zeta\in \Pi(B_x,B_y)$.
	There is a unique geodesic between $x$ and $y$, which must pass through $xs$ and $yt$ for some $s, t\in S$. 
	For all other pairs $a, b\in S\setminus\{s, t\}$, the distance from $xa$ and $yb$ is $d(x,y)+2$, as the geodesic from $xa$ to $yb$ first goes through $x$, travels along the unique geodesic from $x$ to $y$, then travels to $yb$.
	With this in mind, there are essentially two type of transport plan available between $B_x$ and $B_y$: 
	The first sends $xs$ to $yt$, and then arbitrarily sends each $xa$ to a choice of $yb$, giving a total cost of 
	$$\frac{1}{n}(d(x,y)-2+(n-1)(d(x,y)+2))=\frac{1}{n}(nd(x,y)+2(n-2))=d(x,y)+2-\frac{2}{n}$$.
	The second type sends $xs$ to some $ys'$, $s'\neq t$ and sends some other $xt'$ to $yt$, then arbitrarily sends all other points in $B_x$ to points in $B_y$.
	The distance from $xs$ to $ys'$ and from $xt'$ to $yt$ are both equal to $d(x,y)$, and the distance for all other choices is $d(x,y)+2$, again giving us a total cost of $d(x,y)+2-\frac{2}{n}$ 
	(note that $\GenCon(x,y)$ is one of the possible plans given by this second protocol).
	In either case, the total cost is equal to the cost of the plan given by $\GenCon(x,y)$.
\end{example}

In both the above cases we have shown that Ollivier's curvature agrees with comparison curvature. 
A naive hope is that this would hold true for all finitely generated groups. However, it is not difficult to find a counterexample:

\begin{example}
	Let $G=\langle s,t | s^2=t^2=1,sts=tst\rangle = S_3$.
	$$\GenCon(s,e)= \frac{1}{2}(|sss| + |tst|) = 2$$
	In comparison, let $\sigma(s)=t$, $\sigma(t)=s$.
	Consider the transport plan $\zeta$ that sends $xs$ to $y\sigma(s)$ for $x,y\in G$.
	The cost of $\zeta(B_s,B_e)$ equals $$\frac{1}{2}(d(s, st) + d(t, ss)) = 1$$
	Thus by choosing a different transport plan, we were able to lower the transport cost from the one given by $\GenCon(s,e)$.
\end{example}

However, one may hope that the optimal permutation does not depend on the choice of $g$ for which we are measuring the curvature. 
However, the previous example shows that this is not the case. 
Indeed, if we repeat the above calculations for plans between $sts$ and $e$, we find that $\GenCon(sts,e)$ is the optimal transportation plan and not $\zeta$. 

Thus for each group $G$ with generating set $S$ we get a function $\phi_S\colon G\to \sym(S)$ which sends an element $g$ to the permutation giving the discrete Ricci curvature $\kappa(g)$. 
More generally we can consider the sequence of functions $\phi_S^r\colon G\to \sym(B_r)$ which gives the permutation giving $\kappa_r^{\mathcal{B}}(g)$ and $\psi_S^r\colon G\to \sym(S_r)$ which gives the permutation giving $\kappa_r^{\mathcal{S}}(g)$. 
Note that there could be several permutation realising the curvature of a group element. 
For instance in the free group we can take $\psi_S^1(a) = e$ or the permutation which interchanges $b$ and $b^{-1}$, these both realise the curvature of $a$. 
We pose some questions relating to these function. 
\begin{itemize}
	\item For which groups is $\phi_S^r$ the trivial map? These are the groups for which $\GenCon$ is the optimal transport plan. 
	\item Are there any groups for which $\phi_S^r$ are non-trivial homomorphisms?
	\item Is $\phi_S^r = \prod_{i=0}^r \psi_S^i$? 
	\item If the above fails, is it still the case (as in $\GenCon$) that $\phi_S^r$ preserves the sphere of radius $l$ for all $l\leq r$?
\end{itemize}

\bibliographystyle{alpha}
\bibliography{bib}

\begin{thebibliography}{BNDK17}

\bibitem[AM19]{alek}
Ilya Alekseev and Ruslan Magdiev.
\newblock The language of geodesics for the discrete heisenberg group, 2019.

\bibitem[BH09]{bridson}
Martin Bridson and Andr\'e Haefliger.
\newblock {\em Metric Spaces of Non-Positive Curvature}, volume 319.
\newblock Springer-Verlag, 01 2009.

\bibitem[BM19]{brezi}
Haim Brezis and Petru Mironescu.
\newblock The plateau problem from the perspective of optimal transport.
\newblock {\em Comptes Rendus Mathematique}, 357(7):597 -- 612, 2019.

\bibitem[BNDK17]{medscale}
Assaf Bar-Natan, Moon Duchin, and Robert Kropholler.
\newblock Medium-scale curvature for cayley graphs, 2017.

\bibitem[CT05]{taback}
Sean Cleary and Jennifer Taback.
\newblock {Dead end words in lamplighter groups and other wreath products}.
\newblock {\em The Quarterly Journal of Mathematics}, 56(2):165--178, 06 2005.

\bibitem[GdlH90]{ghys}
\'{E}. Ghys and P.~de~la Harpe, editors.
\newblock {\em Sur les groupes hyperboliques d'apr\`es {M}ikhael {G}romov},
  volume~83 of {\em Progress in Mathematics}.
\newblock Birkh\"{a}user Boston, Inc., Boston, MA, 1990.
\newblock Papers from the Swiss Seminar on Hyperbolic Groups held in Bern,
  1988.

\bibitem[Gro87]{gromov}
M.~Gromov.
\newblock {\em Hyperbolic Groups}, pages 75--263.
\newblock Springer New York, New York, NY, 1987.

\bibitem[LEH09]{lehnert}
JÖRG LEHNERT.
\newblock Some remarks on depth of dead ends in groups.
\newblock {\em International Journal of Algebra and Computation},
  19(04):585--594, 2009.

\bibitem[NW19]{expgrowth}
Thang Nguyen and Shi Wang.
\newblock Cheeger-gromoll splitting theorem for groups, 2019.

\bibitem[Oll07]{ollivier}
Yann Ollivier.
\newblock Ricci curvature of metric spaces.
\newblock {\em Comptes Rendus Mathematique}, 345(11):643 -- 646, 2007.

\end{thebibliography}

\end{document}